\documentclass[a4paper,12pt]{amsart}
\usepackage[utf8]{inputenc}
\usepackage{amsmath,amssymb,amsthm}
\usepackage[left=3cm, right=3cm]{geometry}
\usepackage{dsfont}
\usepackage{comment}
\usepackage{esint}
\usepackage{color}
\usepackage{tikz}
\usepackage{enumerate}
\usepackage{multicol}
\usepackage{mathtools}
\usepackage[nocompress]{cite}

\allowdisplaybreaks
\numberwithin{equation}{section}

\mathtoolsset{showonlyrefs}
\usetikzlibrary{decorations.pathreplacing}

\newcommand{\R}{\mathbb{R}}
\newcommand{\N}{\mathbb{N}}

\newcommand{\Z}{\mathbb{Z}}

\newcommand{\eps}{\varepsilon}

\renewcommand{\div}{\operatorname{div }}
\newcommand{\curl}{\operatorname{curl }}

\newtheorem{Theorem}{Theorem}[section]

\theoremstyle{definition}

\newtheorem{remark}[Theorem]{Remark}
\newtheorem{example}{Example}
\newtheorem{theorem}[Theorem]{Theorem}
\newtheorem*{theorem*}{Theorem}
\newtheorem{proposition}[Theorem]{Proposition}

\newtheorem{corollary}{Corollary}

\begin{document}

\title{Rotations with constant Curl are constant}

\author[J. Ginster]{Janusz Ginster}
\address{Janusz Ginster\\Institut f\"ur Mathematik \\Humboldt-Universit\"at zu Berlin\\Unter den Linden 6\\ 10099 Berlin, Germany}
\email{janusz.ginster@math.hu-berlin.de}

\author[A. Acharya]{Amit Acharya}
\address{Amit Acharya\\ Dept.~of Civil \& Environmental Engineering and Center for Nonlinear Analysis \\ Carnegie Mellon University \\ 5000 Forbes Avenue \\ Pittsburgh, PA 15213, USA}
\email{acharyaamit@cmu.edu}


\begin{abstract}
\noindent We address a problem that extends a fundamental classical result of continuum mechanics from the time of its inception, as well as answers a fundamental question in the recent, modern nonlinear elastic theory of dislocations. Interestingly, the implication of our result in the latter case is qualitatively different from its well-established analog in the linear elastic theory of dislocations.

It is a classical result that if $u\in C^2(\R^n;\R^n)$ and $\nabla u \in SO(n)$ it follows that $u$ is rigid. In this article this result is generalized to matrix fields with non-vanishing $\curl$. It is shown that every matrix field $R\in C^2(\Omega \subseteq \R^3;SO(3))$ such that $\curl R = constant$ is necessarily constant. Moreover, it is proved in arbitrary dimensions that a measurable rotation field is as regular as its distributional $\curl$ allows. In particular, a measurable matrix field $R: \Omega \to SO(n)$, whose $\curl$ in the sense of distributions is smooth, is also smooth.

\end{abstract}

\maketitle

\tableofcontents

\section{Introduction}
It is a classical result of continuum mechanics, known from the time of the brothers Cosserat (1896) (according to Shield \cite{RTS73}), that if a $C^2$ deformation of a connected domain $\Omega \subset \mathbb{R}^3$ given by $y: \Omega \to \mathbb{R}^3$ with deformation gradient $\nabla y =: F$ has a constant Right Cauchy-Green tensor field, i.e., $F^TF =$ constant, then $y$ is a homogeneous deformation, i.e., $F=$ constant. Shield \cite{RTS73} gave an elegant proof (with references to other proofs by Forsyth, and Thomas) whose hypothesis was marginally weakened in\cite{JAB89}. An elementary proof using ideas from classical Riemannian Geometry arises from considering parametrizations of $\Omega$ and $y(\Omega)$ in a Rectangular Cartesian coordinate system. Then the condition $F^TF =$ constant allows associating spatially constant metric tensor component fields on the two patches; a use of Christoffel's transformation rule for the Christoffel symbols then yields $\nabla^2 y = 0$. This result implies that if the deformation gradient field of a deformation is known to be `pointwise rigid,' i.e., $\nabla y(x) = F(x) \in SO(3) \, \forall x \in \Omega$, then $F =$ constant $\in SO(3)$,  and the deformation $y$ is globally rigid. First generalizations of this result go back to Reshetnyak who proved in \cite{Re67} that if $y_k \rightharpoonup y$ in $W^{1,2}$ and $\operatorname{dist}(\nabla y_k,SO(n)) \to 0$ in measure then $\nabla y$ is necessarily a constant rotation. A proof of this result using Young measures can be found in \cite{JK88}. John proved in \cite{Jo61} that if $y\in C^1$ and $\operatorname{dist}(\nabla y,SO(n)) \leq \delta$ for a sufficiently small $\delta >0$ then $[\nabla y]_{BMO} \leq C(n) \delta$. Without the assumption that $\nabla y$ is uniformly close to $SO(n)$, Kohn proved optimal bounds for $\min_{R \in SO(n), b \in \R^n} \| y - (Rx + b) \|_{L^p}$ (but not for $\|\nabla y - R\|_{L^p}$) in \cite{Ko82}. Optimal bounds on $\nabla y -R$ in $L^2$ were derived in the celebrated work of Friesecke, James, and M\"{u}ller,  \cite{FJM02}. The authors prove that for an open, connected domain with Lipschitz boundary $\Omega \subseteq \R^n$ there exists $C(\Omega)>0$ such that for every $y \in W^{1,2}(\Omega;R^n)$ there exists a rotation $R\in SO(n)$ satisfying
\begin{equation}\label{eq: rigiditygradients}
\int_{\Omega} |\nabla y - R|^2 \, dx \leq C(\Omega) \int_{\Omega} \operatorname{dist}(\nabla y,SO(n))^2 \, dx.
\end{equation}
As pointed out in \cite{CoSc06} $L^p$-versions of the above estimate also hold for $1<p<\infty$.  Generalizations to interpolation spaces were established in \cite{CoDoMu14}. 

Regardless of the smoothness hypotheses involved, all of the above results crucially rely on the fact that the field $F$ is the gradient of some deformation $y$. Going beyond the realm of deformations, it seems natural to interpret the global rigidity question in the following way: Let  $R \in C^1(\Omega; SO(3))$ be specified with $\curl R = 0$ in $\Omega$; then $R =$ constant. Posed in this manner, it seems natural to ask whether the hypothesis $\curl R = 0$ is optimal or whether it can be further weakened. It is this question that is dealt with in this paper with an affirmative answer. Specifically, we show that global rigidity is obtained even for $\curl R =$ constant on $\Omega$. This result, for $\Omega \subseteq \mathbb{R}^2$ and $R \in C^2(\Omega;SO(2))$, was obtained in \cite{A19}. Here, we prove it for $R: \Omega \to SO(3)$ merely measurable. This three-dimensional result is based on significantly different ideas from \cite{A19}, and generates also a different proof for the 2-d case. 

Rigidity estimates similar to \eqref{eq: rigiditygradients} for non-gradient fields were first established in the linear theory and dimension $2$ in \cite{GaLePo10}. The nonlinear analogue was proved in  \cite{MSZ14}. It reads as follows: For $\Omega \subseteq \R^2$ open and connected with Lipschitz boundary there exists $C(\Omega)>0$ such that for every $F \in L^2(\Omega;\R^{2\times 2})$ such that $\curl F$ is a bounded measure there exists $R \in SO(2)$ satisfying
\begin{equation}\label{eq: rigidityincompatible}
\int_{\Omega} |F-R|^2 \, dx \leq C(\Omega) \left( \int_{\Omega} \operatorname{dist}(F,SO(2))^2 \, dx + |\curl F|(\Omega)^2 \right).
\end{equation}
For a version with mixed growth, see \cite{Gi19}. A generalization to higher dimensions was established in \cite{LL16}. Clearly, a rigidity estimate like \eqref{eq: rigidityincompatible} does not directly imply that rotation fields with a constant but non-zero $\curl$ are constant as the estimate \eqref{eq: rigidityincompatible} applied to a field with a constant $\curl$ does not provide more information than the same estimate applied to a field with a bounded but non-constant $\curl$. However, there are obviously non-constant rotation fields with a bounded $\curl$. Therefore, the proof of our result will be based on a different approach (see section \ref{sec: 2d} for the idea of the proof and its connection to the gradient setting).
Instead, the rigidity estimate \eqref{eq: rigidityincompatible} can be used to prove higher regularity for rotation fields, see Section \ref{sec: regularity}, whereas our rigidity result is based on a PDE approach, see Sections \ref{sec: 2d} and \ref{sec: 3d}.

It turns out that the question raised above is of relevance in the theory of dislocations, as explained in detail in \cite{A19}, with connections to the linear elastic theory of dislocations. Briefly, considering a nonlinear elastic material with a `single-well' elastic energy density, our result shows that a traction-free body with a constant (non-vanishing) dislocation density cannot be stress-free (such a field is computed in \cite[Sec. 5.3]{AZA20}). This is in stark contrast to the linear theory of dislocations in which the same body under identical hypotheses would necessarily be stress-free. An interesting question in this interpretation of our work is the characterization of the resulting stress field in a material with a `multiple-well' energy density, in particular, whether a stress-free state can arise for a constant dislocation density. \\

This article is organized as follows. First we introduce the needed notation. Then we prove that a regular rotation field with a constant $\curl$ is constant in dimension $2$ (Section \ref{sec: 2d}) and $3$ (Section \ref{sec: 3d}). In Section \ref{sec: regularity} we prove regularity of rotation fields in terms of the regularity of its $\curl$. This shows that the results proved in Section \ref{sec: 2d} and Section \ref{sec: 3d} apply more generally to measurable rotation fields with a constant $\curl$ in the sense of distributions.

\section{Notation}
Throughout the whole article we use the Einstein summation convention i.e., we sum over indices that appear twice.

Moreover, we denote by $Id$ the identity matrix in any dimension. For a matrix $A$ we write $A_i$ for its $i$-th row. For the set of rotations in $\R^n$ we write $SO(n) = \{R \in \R^{n\times n}: A^T A = Id, \det(A) = 1 \}$. The trace of a matrix $A \in \R^{n\times n}$ is given by $\mathrm{tr}(A) = \sum_{k=1}^n A_{kk}$, the scalar product between two matrices $A,B \in \R^{n\times n}$ is given by $A:B = \mathrm{tr}(A^TB)$. For a matrix $A \in \R^{n\times n}$ we write $A_{sym} = \frac12 (A+A^T)$ and $A_{skew} = \frac12( A - A^T)$. The spaces of symmetric or skew-symmetric matrices are denoted by $Sym(n) = \{A \in \R^{n\times n}: A^T = A\}$ and $Skew(n)= \{A\in \R^{n\times n}: A^T = -A\}$, respectively. For two vectors $a,b \in \R^3$ the cross product $a \times b \in \R^3$ is defined as usual as $(a\times b)_i = \eps_{ijk} a_j b_k$. Here, $\eps_{ijk}$ is the sign of the permutation $(ijk)$.

Let $\Omega \subseteq \R^n$  and connected. Throughout the whole paper we use standard notation for the space of $k$-times differentiable functions from $\Omega$ to $\R^m$, $C^k(\Omega;\R^m)$, the space of $p$-integrable functions (more precisely, equivalence classes of these functions) on $\Omega$ with values in $\R^m$, $L^p(\Omega;\R^m)$, Sobolev spaces, $W^{k,p}(\Omega;\R^m)$, and the space of vector-valued Radon-measures, $\mathcal{M}(\Omega;\R^m)$. For a vector-valued Radon measure $\mu$ we denote by $|\mu|$ its total variation measure. The space of functions of bounded variation $BV(\Omega;\R^m)$ consists of function $f\in L^1(\Omega;\R^m)$ whose weak derivative is a vector-valued Radon measure with finite total variation i.e., there exists $\mu \in \mathcal{M}(\Omega;\R^{n\times m})$ with $|\mu|(\Omega) < \infty$ such that for all $\varphi \in C^{\infty}_c(\Omega;\R^m)$ and $i\in\{1,\dots,n\}$ it holds
\[
\int_{\Omega} u \cdot \partial_i \varphi \, dx = - \int_{\Omega} \varphi \cdot d\mu_i.
\]
In this case we write $Du = \mu$.

In addition we recall quickly standard notation for classical differential operators. The divergence operator for a vector field $f=(f_1,\dots,f_n)$ on a subset of $\R^n$ is given by $\div(f) = \sum_{k=1}^n \partial_k f_k$. For a vector field on a subset of $\R^2$ we write $\curl(f) = \partial_1 f_2 - \partial_2 f_1$, for a vector field $f$ on a subset of $\R^3$ the $i$-th component of the vector field $\curl(f)$ is given by $\curl(f)_i = \eps_{ijk} \partial_j f_k$. For arbitrary $n \in \N$ we generalize this notation to $\operatorname{Curl}(f) = \left( \partial_j f_k - \partial_k f_j \right)_{j,k = 1}^n$. In dimension $2$ and $3$ the notions $\curl$ and $\operatorname{Curl}$ can easily be identified. For matrix fields $\operatorname{Curl}$, $\div$ and $\curl$ will always be applied rowwise.

We recall that for a function $f \in L^1_{loc}(\Omega;\R^n)$ we say that for $\mu \in \mathcal{M}(\Omega; \R^{n\times n})$ it holds $\operatorname{Curl}(f) = \mu$ in the sense of distributions if we have for all $\varphi \in C^{\infty}_c(\Omega;\R)$
\[
\int_{\Omega} f_k \partial_j \varphi - f_j \partial_k \varphi_k \, dx = - \int_{\Omega} \varphi \, d\mu_{jk}.
\]
Note that a function $\alpha \in L^1_{loc}(\Omega;\R^m)$ can always be associated to a vector-valued Radon measure $\mu \in \mathcal{M}(\Omega;\R^m)$ through $\mu_{\alpha}(A) = \int_A \alpha(x) \, dx$. For $f, \alpha \in L^1_{loc}$ we also write $\operatorname{Curl} f = \alpha$ instead of $\operatorname{Curl} f = \mu_{\alpha}$.

\section{Rigidity for Rotation Fields in Dimension $\mathbf{2}$} \label{sec: 2d}

We start by reconsidering the case $n=2$. In \cite{A19} it was shown that a function $R \in C^2(\Omega;SO(2))$ such that $\curl R$ is constant is necessarily constant. 
In this section we give an alternative proof to this statement which uses the idea of the proof for gradients. A similar strategy will be used in the three-dimensional setting.

Let us quickly recall the argument for gradients in dimension $n$. Let $R = \nabla u \in C^1(\Omega;SO(n))$ for some $u\in C^2(\Omega;\R^n)$. We note that $\operatorname{cof } \nabla u = \nabla u$, $\div \operatorname{cof }(\nabla u) = 0$ and $|\nabla u|^2 = n$. Thus, $\Delta u = 0$ and $0 = \Delta |\nabla u|^2$. Then one computes $0 = \Delta |\nabla u|^2 = 2 \nabla (\Delta u) : \nabla u + |\nabla^2 u|^2 = |\nabla^2 u|^2$. Consequently, $\nabla u = R$ is constant. 
 
 \begin{theorem}\label{thm: 2d}
  Let $\Omega \subseteq \R^2$ be open and connected. Let $R \in C^2(\Omega ;SO(2))$ and $\alpha \in \R^2$ such that $\curl R = \alpha$. Then $R$ is constant.
 \end{theorem}
\begin{proof}
As $R(x) \in SO(2)$ for all $x \in \Omega$, there exists a $C^2$-vector field $e : \Omega \to \R^2$ such that 
 \[
  e_1(x)^2 + e_2(x)^2 = 1 \text{ and } R(x) = \begin{pmatrix}
                                                   e_1(x) && e_2(x) \\ -e_2(x) && e_1(x)
                                                  \end{pmatrix} \text{ for all } x\in \Omega.
 \]
As $\curl R = \alpha$, we find that
\begin{align*}
 \partial_1 e_2 - \partial_2 e_1 &= \alpha_1, \\
 \partial_1 e_1 + \partial_2 e_2 &= \alpha_2,
\end{align*}
from which we derive
\begin{align*}
 \partial_1 \partial_1 e_2 - \partial_1 \partial_2 e_1 = 0, \\
 \partial_2 \partial_1 e_2 - \partial_2 \partial_2 e_1 = 0, \\
 \partial_1 \partial_1 e_1 + \partial_1 \partial_2 e_2 = 0, \\
 \partial_2 \partial_1 e_1 + \partial_2 \partial_2 e_2 = 0.
\end{align*}
Adding the fourth to the first equation and subtracting the second from the third equation we find that
\[
 \Delta e_1 = \Delta e_2 = 0.
\]
Using that $e_1(x)^2 + e^2(x) = 1$ for all $x \in \Omega$, we obtain
\[
 0 = \Delta(e_1^2 + e_2^2) = 2 e_1 \Delta e_1 + 2 |\nabla e_1|^2 + 2e_2 \Delta e_2 + 2|\nabla e_2|^2 = |\nabla e_1|^2 + |\nabla e_2|^2.
\]
As $\Omega$ is connected this implies that $e$ (and consequently $R$) is constant.
\end{proof}

In view of Theorem \ref{thm: 2d} we see that the generalized rigidity estimate \eqref{eq: rigidityincompatible} does not provide the optimal estimate for rotation fields with a constant $\curl$.  The na\"{i}ve extension of the generalized rigidity estimate \eqref{eq: rigidityincompatible} incorporating the result of Theorem \ref{thm: 2d} would allow the subtraction of a constant from the $\curl$ on the right hand side: For every open, bounded and connected set $\Omega \subseteq \R^2$ with Lipschitz boundary there exists $C(\Omega)>0$ such that for every $F\in L^2(\Omega;\R^{2\times 2})$ with $\curl F \in \mathcal{M}(\Omega;\R^2)$ and $\alpha \in \R^2$ there exists $R\in SO(2)$ satisfying
\[
\int_{\Omega} |F - R|^2 \, dx \leq C(\Omega) \left( \int_{\Omega} \operatorname{dist}(F,SO(2))^2 \, dx + |\curl(F) - \mu|(\Omega)^2\right),
\]
where $\mu = \alpha \, \mathcal{L}^2$. 

However, the following example shows that a statement of this type cannot be true as it does not hold true in the linearized setting, c.f. the discussion in \cite{A19}.

\begin{example}
Let $\Omega = B_1(0)$. For $\eps>0$ we define $F_{\eps}: \Omega \to \R^{2\times 2}$ by 
\[
F_{\eps}(x) = Id + \eps \begin{pmatrix} 0 && x_1 \\ -x_1 && 0 
\end{pmatrix}.
\]
First we notice that $\curl F_{\eps} = \eps \begin{pmatrix} 1\\0\end{pmatrix}$.
Next, we observe that $\fint_{\Omega} F_{\eps} \, dx = Id$ and therefore $\int_{\Omega} |F_{\eps} - Id|^2 \, dx \leq \int_{\Omega} |F_{\eps} - R|^2 \, dx$ for all $R\in SO(2)$. Now, we compute $\int_{\Omega} |F_{\eps} - Id|^2 \, dx = \int_{\Omega} 2\eps^2 x_1^2 \, dx = \frac{\pi}2 \eps^2$. On the other hand, a second order Taylor expansion at $Id$ shows that
\[
\operatorname{dist}(F_{\eps}(x),SO(2))^2 \leq |(F_{\eps}(x) - Id)_{sym}|^2 + C |F_{\eps} - Id|^3 \leq C \eps^3.
\]
Consequently, $\int_{\Omega} \operatorname{dist}(F_{\eps}(x),SO(2))^2 \, dx \leq C \eps^3$. In particular we see that there cannot exist a constant $C(\Omega) > 0$ such that for every $\eps >0$ there exists $R_{\eps} \in SO(2)$ satisfying
\[
\int_{\Omega} |F_{\eps} - R_{\eps}|^2 \, dx \leq C(\Omega) \left( \int_{\Omega} \operatorname{dist}(F_{\eps}, SO(2))^2 \, dx + \left(\left|\curl(F) - \eps \begin{pmatrix} 1 \\ 0 \end{pmatrix}\right|(\Omega)\right)^2 \right).
\]
\end{example}

\section{Rigidity for Rotation Fields in Dimension $\mathbf{3}$} \label{sec: 3d}

This section is devoted to prove that in three dimensions a rotation field whose $\curl$ is constant has to be locally constant.

\subsection{A Simple Argument for $\mathbf{\Omega = \R^3}$}
We start with a simple argument for $\Omega = \R^3$ which is based on Stokes' theorem.

\begin{theorem}
Let $R \in C^1(\R^3;SO(3))$ such that $\curl R = \alpha$ for some $\alpha \in \R^{3\times 3}$. Then $\alpha = 0$ and $R$ is constant.
\end{theorem}
\begin{proof}
If $\alpha = 0$ then the result follows by the classical rigidity result for gradients. So we assume that $\alpha \neq 0$. Hence, there exists $v \in \R^3$ such that $\alpha v \neq 0$. Up to a rotation we may assume that $v = \begin{pmatrix} 0 \\ 0 \\ 1 \end{pmatrix}$. Now, we define for $\rho > 0$ the two-dimensional disk and circle with radius $\rho$ as
\begin{align}
&D^{(2)}_{\rho} = \left\{ (x_1,x_2,x_3) \in \R^3: x_1^2 + x_2^2 < \rho^2, x_3=0 \right\} \\ 
\text{ and } &S_{\rho}^{(2)} = \left\{ (x_1,x_2,x_3) \in \R^3: x_1^2 + x_2^2 = \rho^2, x_3 = 0 \right\}.
\end{align}
We choose $v$ to be the normal to $D_{\rho}^{(2)}$ and denote by $\tau \in S^2$ the corresponding positively oriented tangent to $S^{(2)}_{\rho}$.
Using Stokes' theorem we compute
\begin{align}
\pi \rho^2 \, \left\|\alpha v\right\| = \left\|\int_{D^{(2)}_{\rho}} \curl R \cdot \nu \, \mathcal{H}^2\right\| = \left\|\int_{S^{(2)}_{\rho}} R \tau \, d\mathcal{H}^1\right\| \leq 2\pi \rho.
\end{align}
For the last inequality we used that $\| R \tau \| = 1$ since $R \in SO(3)$. This yields a contradiction for every $\rho >  \frac2{\| \alpha v\|}$.
\end{proof}

\begin{remark}
The proof shows that there cannot be $R \in C^1(\Omega;SO(3))$ with $\curl R = \alpha$ and  $B_{2\|\alpha\|_{op}+\delta}(x) \subseteq \Omega$ for some $x \in \Omega$, $\delta > 0$ and $\| \alpha \|_{op} = \sup\{ \alpha v: \|v\| = 1\}$. 
\end{remark}

\subsection{The General Result}

In this section we prove our main result, namely that on any open and connected set $\Omega \subseteq \R^3$ every sufficiently regular function $R: \Omega \to SO(n)$ with a constant $\curl$ is constant. 

Our approach is quite similar to the proof of Theorem \ref{thm: 2d}, namely we first show that a field of rotations $R:\Omega \to \R^3$ satisfies a linear elliptic PDE. Together with the assumption that $\curl R$ is constant this will yield an equality for $|\nabla R|^2$ in terms of $R$ and $\curl R$.

Before we prove the main result we collect a few results that will be needed later.

\begin{proposition}\label{prop: collection3d}
Let $\Omega \subseteq \R^3$ be open and $R \in C^2(\Omega; SO(3))$ with $\curl R = \alpha$ for some constant matrix $\alpha \in \R^{3\times 3}$. Then the following hold:
\begin{enumerate}[(i)]
\item $\div R_i = \varepsilon_{ijk} \, \alpha_j \cdot R_k$ for $i\in\{1,2,3\}$. \label{item: collecti}
\item $\Delta R_i = \varepsilon_{ijk} \nabla (\alpha_j \cdot R_k)$. \label{item: collectii}
\item $|\nabla R|^2 =- \mathrm{tr} (R^T \alpha R^T \alpha)$.\label{item: collectiii}
\item $\mathrm{tr}(R^T \alpha R^T \alpha) = |(R^T \alpha)_{sym}|^2 - |(R^T \alpha)_{skew}|^2$. \label{item: collectiv}
\item If $R(x_0) = Id$ then $|\div(R)(x_0)|^2 = 2 |\alpha_{skew}|^2$. \label{item: collectv}
\item $\sum_{i=1}^3 |(\nabla R_i)_{sym}|^2 \geq \frac13 |\div(R)|^2$. \label{item: collectvi}
\item $\sum_{i=1}^3 |(\nabla R_i)_{skew}|^2 = \frac12 |\alpha|^2$. \label{item: collectvii}
\end{enumerate}
\end{proposition}
\begin{proof}
As $R$ takes values in $SO(3)$ we note that the rows of $R$ form an orthonormal frame. Hence, for $i\in \{1,2,3\}$ we have
\[
2 R_i = \varepsilon_{ijk}\, R_j \times R_k.
\]
Consequently, we can compute $\div R_i$ as follows
\begin{align}
2\, \div(R_i) = \varepsilon_{ijk} \div\left( R_j \times R_k \right) &= \varepsilon_{ijk} \left( \curl(R_j) \cdot R_k - R_j \cdot \curl (R_k) \right) \\
& = \varepsilon_{ijk} \,(\alpha_j \cdot R_k - R_j \cdot \alpha_k) \\
&= 2 \varepsilon_{ijk} \, \alpha_j \cdot R_k.
\end{align}
This shows \eqref{item: collecti}. Now we recall the well-known identity $\curl \curl = -\Delta + \nabla \div$. As $\curl R$ is constant, this yields
\begin{equation}\label{eq: laplaceR}
0 = -\Delta R_i + \nabla \div R_i,
\end{equation}
which shows in combination with \eqref{item: collecti} claim \eqref{item: collectii}. For \eqref{item: collectiii} we first observe for $i \in \{1,2,3\}$ that
\[
0 = \Delta (|R_i|^2) = 2 \Delta(R_i) \cdot R_i + 2 |\nabla R_i|^2.
\]
In combination with \eqref{item: collectii} and \eqref{eq: laplaceR} this implies
\begin{align} \label{eq: normgradientR}
-|\nabla R|^2 = \eps_{ijk} \nabla (\alpha_j \cdot R_k) \cdot R_i = \eps_{ijk}  \,\alpha_{jl} \,  \left(\partial_m R_{kl}\right) R_{im}.
\end{align}
Next, we use \eqref{eq: representationDR2} i.e., we have for $m,k,l \in \{1,2,3\}$
\[
2 \left(\partial_m R\right)_{kl} = \varepsilon_{rml} \alpha_{kr}  + \varepsilon_{rsl} R_{ks} \left( R^T \alpha \right)_{mr} + \varepsilon_{rsm} R_{ks} \left( R^T \alpha \right)_{lr}.\]
Plugging this identity into \eqref{eq: normgradientR} yields
\begin{align}
2 \eps_{ijk}  \,\alpha_{jl} \,  \left(\partial_m R_{kl}\right) R_{im} = &\eps_{ijk}  \,\alpha_{jl} \,  R_{im} \, \varepsilon_{rml} \alpha_{kr}  \\ 
&+ \eps_{ijk}  \,\alpha_{jl} \,  R_{im} \, \varepsilon_{rsl} R_{ks} \left( R(x)^T \alpha \right)_{mr} \\
&+ \eps_{ijk}  \,\alpha_{jl} \,  R_{im} \varepsilon_{rsm} R_{sr} \left( R^T \alpha \right)_{lr} \\
 =:& (I) + (II) + (III).
\end{align}
No we compute 
\begin{align*}
(I) &= \eps_{ijk}  \,\alpha_{jl} \,  R_{im} \, \varepsilon_{rmq} \alpha_{kr} = \eps_{ijk}  \,\alpha_{jl} \, (\alpha_k \times R_i)_{l} = \eps_{ijk} \, (\alpha_k \times R_i) \cdot \alpha_j, \\
(II) &= \eps_{ijk}  \,\alpha_{jl}  \, \varepsilon_{rsl} R_{ks} \,  \alpha_{ir} \\ &= \eps_{ijk} \, (R_k \times \alpha_j)_r \alpha_{ir} =\eps_{ijk} \, (R_k \times \alpha_j) \cdot \alpha_i = -\eps_{ijk} (\alpha_j \times R_k) \cdot \alpha_i = -(I), \\
(III) &= \eps_{ijk}  \,\alpha_{jl} \,  R_{im} \varepsilon_{rsm} R_{ks} \left( R^T \alpha \right)_{lr} \\
&= \eps_{ijk} \, \alpha_{jl} \, (R_k \times R_i)_{r} \, \left( R^T \alpha \right)_{lr} \\
&= \eps_{ijk} (R_k \times R_i)_r\, (\alpha R^T \alpha)_{jr} \\ &= 2 R_{jr} \, (\alpha R^T \alpha)_{jr} = 2 (R^T \alpha R^T \alpha)_{rr} = \mathrm{tr}(R^T\alpha R^T \alpha).
\end{align*}
Combining \eqref{eq: normgradientR}, (I), (II) and (III) yields \eqref{item: collectiii}. 

For \eqref{item: collectiv} we simply compute
\begin{align*}
\mathrm{tr}(R^T\alpha R^T\alpha) = &(R^T\alpha)^T : (R^T\alpha) \\ = &\left((R^T\alpha)_{sym} - (R^T\alpha)_{skew} \right): \left( (R^T\alpha)_{sym} + (R^T\alpha)_{skew} \right) \\  = &\left| (R^T \alpha)_{sym} \right|^2 - \left| (R^T \alpha)_{skew} \right|^2.
\end{align*}
Next, we assume that $R(x_0) = Id$. By \eqref{item: collecti} we have that $\div(R_i)(x_0) = \eps_{ijk} \, \alpha_j \cdot R_k(x_0) = \eps_{ijk} \alpha_{jk}$. Consequently,
\[
\alpha_{skew} = \frac12 \begin{pmatrix}
0 && \div(R_3)(x_0) && -\div(R_2)(x_0) \\ -\div(R_3)(x_0) && 0 && \div(R_1)(x_0) \\ \div(R_2)(x_0) && -\div(R_1)(x_0) && 0
\end{pmatrix}
\]
and therefore $|\alpha_{skew}|^2 = 2 |\div(R)(x_0)|^2$, which is \eqref{item: collectv}. \\
For \eqref{item: collectvi}, we estimate
\begin{align}
\sum_{i=1}^3 |(\nabla R_i)_{sym}|^2 &\geq \sum_{i=1}^3 \left( (\nabla R_i)_{11}^2 + (\nabla R_i)_{22}^2 + (\nabla R_i)_{33}^2 \right) \\
&\geq \sum_{i=1}^3 \frac13 \left(\mathrm{tr}(\nabla R_i)\right)^2 \\
&= \frac13 \sum_{i=1}^3 \left( \div(R_i) \right)^2 = \frac13 |\div(R)|^2. 
\end{align}
Eventually, we prove \eqref{item: collectvii}. We observe for $i \in \{1,2,3\}$
\begin{align}
(\nabla R_i)_{skew} &= \frac12 \begin{pmatrix}
0 && \partial_2 R_{i1} - \partial_1 R_{i2} && \partial_3 R_{i1} - \partial_1 R_{i3} \\ \partial_1 R_{i2} - \partial_2 R_{i1} && 0 && \partial_3 R_{i2} - \partial_2 R_{i3} \\ \partial_1 R_{i3} - \partial_3 R_{i1} && \partial_2 R_{i3} - \partial_3 R_{i2} && 0
\end{pmatrix} \\
&= \frac12 \begin{pmatrix}
0 && -\alpha_{i3} && \alpha_{i2} \\ \alpha_{i3} && 0 && -\alpha_{i1} \\ -\alpha_{i2} && \alpha_{i1} && 0
\end{pmatrix}.
\end{align}
Therefore,
\begin{equation}
\sum_{i=1}^3 \left| (\nabla R_i)_{skew} \right|^2 = \sum_{i=1}^3 \frac12 |\alpha_i|^2 = \frac12 |\alpha|^2.
\end{equation}
\end{proof}

Armed with the results from Proposition \ref{prop: collection3d} we can now show that every field of rotations with a constant $\curl$ has to be locally constant.

\begin{theorem}\label{thm: 3d}
Let $\Omega \subseteq \R^3$ open and connected, and $R \in C^2(\Omega;SO(3))$ such that $\curl R = \alpha$ for some $\alpha \in R^{3\times 3}$. Then $R$ is constant.
\end{theorem}
\begin{proof}
We assume first that $\Omega$ is simply-connected.
For $\alpha = 0$ the result is the well-known result for gradients. Hence, it suffices to prove that $\alpha = 0$.
Now, let $x_0 \in \Omega$. We may assume that $R(x_0) = Id$. Otherwise consider $\tilde{R}(x) = R(x_0)^T R(x)$ and $\tilde{\alpha} = R(x_0)^T \alpha$. By Proposition \ref{prop: collection3d} \eqref{item: collectiii} and \eqref{item: collectiv} we have
\begin{equation}\label{eq: rigidity3d1}
|\nabla R|^2 = |(R^T \alpha)_{skew}|^2 - |(R^T \alpha)_{sym}|^2.
\end{equation}
On the other hand, combining Proposition \ref{prop: collection3d} \eqref{item: collectvi} and \eqref{item: collectvii} yields
\begin{equation}\label{eq: rigidity3d2}
|\nabla R|^2 = \sum_{i=1}^3 |(\nabla R_i)_{sym}|^2  + |(\nabla R_i)_{skew}|^2 \geq \frac13 |\div(R)|^2 + \frac12 |\alpha|^2.
\end{equation}
Using Proposition \ref{prop: collection3d} \eqref{item: collectv} we find from combining \eqref{eq: rigidity3d1} and \eqref{eq: rigidity3d2} at the point $x_0$
\begin{equation}
|\alpha_{skew}|^2 - |\alpha_{sym}|^2 \geq \frac23 |\alpha_{skew}|^2 + \frac12 |\alpha|^2 = \frac76 |\alpha_{skew}|^2 + \frac12|\alpha_{sym}|^2.
\end{equation}
This implies that $\alpha_{skew} = \alpha_{sym} = 0$ i.e, $\alpha=0$. This completes the proof if $\Omega$ is simply-connected. \\
Eventually we notice that around every point there exists a simply-connected neighborhood which is included in $\Omega$. Then we proved that $R$ is constant in this neighborhood i.e., $R$ is locally constant. As $\Omega$ is connected this implies that $R$ is constant.
\end{proof}

In combination with Corollary \ref{cor: regularity} in Section \ref{sec: regularity} Theorem \ref{thm: 3d} shows our main result.

\begin{theorem}\label{thm: rigiditygeneral3d}
Let $\Omega \subseteq \R^3$ be open and bounded. Then every measurable $R:\Omega \to SO(3)$ with a constant $\curl$ in the sense of distributions is constant.
\end{theorem}

\section{Regularity of Rotation Fields is Dominated by Regularity of Their Curl} \label{sec: regularity}

In this section $\Omega \subseteq \R^n$ denotes an open set. We will show that the regularity of a measurable field $R: \Omega \rightarrow SO(n)$ is determined by the rgularity of its $\operatorname{Curl}$. Precisely, we will show that if $\operatorname{Curl}(R) \in C^k(\Omega;\R^{n\times n \times n})$ for some $k\in \N$ then $R \in C^{k+1}(\Omega; \R^{n\times n})$. In particular, if $\operatorname{Curl}(R)$ is constant then $R$ is smooth.

As a first step we recall a statement from \cite{LL16,LL17}. It states that a field of rotations $R$ whose $\operatorname{Curl}$ is a finite vector-valued Radon measure is already a function of bounded variation. This result was already stated in \cite{LL16,LL17}. The used argument, which we present here for the convenience of the reader, implies local estimates which we will use to derive that $D R$ is absolutely continuous with respect to the measure $\operatorname{Curl}(R)$. The proof uses the generalized rigidity estimate from \cite{MSZ14} for fields with non-vanishing $\curl$ for dimension $2$ and the result from \cite{LL17} for higher dimensions.

\begin{proposition}\label{prop: BV}
 Let $n\geq 2$ and $\Omega \subseteq \R^n$ open and bounded. Then there exists a constant $C>0$ such that for every measurable function $R: \Omega \to SO(n)$ such that $\operatorname{Curl}(R) \in \mathcal{M}(\Omega;\R^{n\times n\times n})$ and $|\operatorname{Curl}R|(\Omega) < \infty$ it holds for every Borel set $A \subseteq \Omega$ that
 \begin{equation}\label{eq: absolutecontinuitycurl}
  |DR|(A) \leq C |\operatorname{Curl } R|(A).
 \end{equation}
In particular, $R \in BV(\Omega;\R^{n\times n})$. 
\end{proposition}
\begin{proof} First, let $A \subseteq \Omega$ be open. For this let $\Omega' \subseteq A$ be open such that $\Omega' \subset \subset \Omega$. For $\delta > 0$ we define 
 \[
  I_{\delta} = \left\{ i \in \delta \Z^n \,|\, i + (-\delta,\delta)^n \subseteq A   \right\}
  \]
  and for $i \in I_{\delta}$ 
  \[ q_i^{\delta} = i + (-\delta/2,\delta/2)^n  \text{ and } Q_i^{\delta} =  i + (-\delta,\delta)^n.
 \]
 Then it holds for $\delta > 0$ small enough that $\Omega' \subseteq \bigcup_{i \in I_{\delta}} q_i^{\delta} \cup N \subseteq A$, where $N \subseteq \Omega$ is a set of Lebesgue measure $0$, see Figure \ref{fig: coveringbycubes}.
 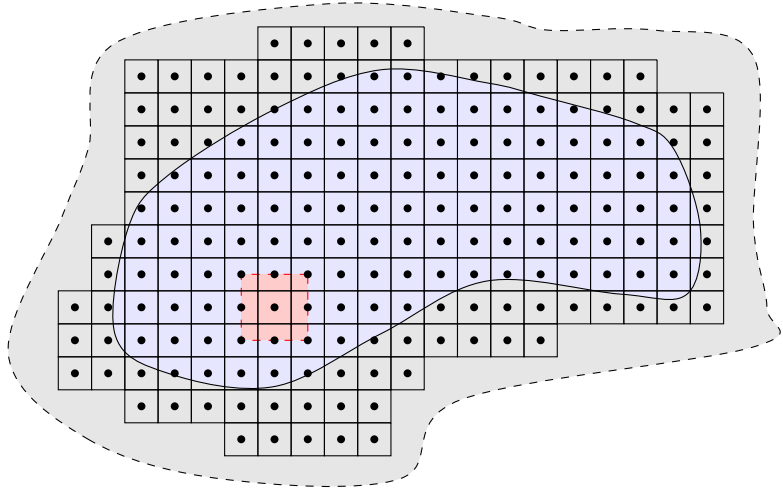
\begin{figure}[t]
\centering

\begin{tikzpicture}[scale =3.5]
\fill[black!10!white]       plot [smooth cycle, tension=0.6] coordinates 
      {(4.0,0.32)(4.3,0.00)(4.8,-0.15)(5.47,-0.15)(5.7,0.15)(6.8,0.35)
       (6.85,0.5)(6.8,0.81)(6.8,1.45)(6.43,1.55)(6.03,1.55)(5.82,1.6)
       (5.2,1.65)(4.4,1.5)(4.3,1.1)(4.2,0.84)};
\fill[blue!10!white] plot [smooth cycle, tension=0.6] coordinates 
      {(4.4,0.4)(4.6,0.25)(5,0.2)(5.4,0.4)(5.8,0.6)(6.3,0.55)(6.55,0.55)
       (6.6,0.8)(6.5,1.1)(6.35,1.2)(6,1.3)(5.8,1.35)(5.4,1.4)(5,1.25)
       (4.6,1)(4.45,0.8)};       
\draw plot [smooth cycle, tension=0.6] coordinates 
      {(4.4,0.4)(4.6,0.25)(5,0.2)(5.4,0.4)(5.8,0.6)(6.3,0.55)(6.55,0.55)
       (6.6,0.8)(6.5,1.1)(6.35,1.2)(6,1.3)(5.8,1.35)(5.4,1.4)(5,1.25)
       (4.6,1)(4.45,0.8)};
\draw[dashed] plot [smooth cycle, tension=0.6] coordinates 
      {(4.0,0.32)(4.3,0.00)(4.8,-0.15)(5.47,-0.15)(5.7,0.15)(6.8,0.35)
       (6.85,0.5)(6.8,0.81)(6.8,1.45)(6.43,1.55)(6.03,1.55)(5.82,1.6)
       (5.2,1.65)(4.4,1.5)(4.3,1.1)(4.2,0.84)};


\fill[red!20!white] (5+0.125,0.5+0.125) -- (5-0.125,0.5+0.125) -- (5-0.125,0.5-0.125) -- (5+0.125,0.5-0.125) -- (5+0.125,0.5+0.125);
\draw[red, dashed] (5+0.125,0.5+0.125) -- (5-0.125,0.5+0.125) -- (5-0.125,0.5-0.125) -- (5+0.125,0.5-0.125) -- (5+0.125,0.5+0.125);

\foreach \i in {4.5,4.625,...,6.625}{
\foreach \j in {0.5,0.625,...,1.25}{
\fill[black] (\i,\j) circle (0.015cm);
\draw (\i+0.0625,\j+0.0625) -- (\i-0.0625,\j+0.0625) -- (\i-0.0625,\j-0.0625) -- (\i+0.0625,\j-0.0625) -- (\i+0.0625,\j+0.0625);
}
}

\foreach \i in {4.375}{
\foreach \j in {0.5,0.625,...,0.75}{
\fill[black] (\i,\j) circle (0.015cm);
\draw (\i+0.0625,\j+0.0625) -- (\i-0.0625,\j+0.0625) -- (\i-0.0625,\j-0.0625) -- (\i+0.0625,\j-0.0625) -- (\i+0.0625,\j+0.0625);
}
}

\foreach \i in {4.5,4.625,4.75,...,6.375}{
\foreach \j in {1.375}{
\fill[black] (\i,\j) circle (0.015cm);
\draw (\i+0.0625,\j+0.0625) -- (\i-0.0625,\j+0.0625) -- (\i-0.0625,\j-0.0625) -- (\i+0.0625,\j-0.0625) -- (\i+0.0625,\j+0.0625);
}
}

\foreach \i in {5,5.125,...,5.5}{
\foreach \j in {1.5}{
\fill[black] (\i,\j) circle (0.015cm);
\draw (\i+0.0625,\j+0.0625) -- (\i-0.0625,\j+0.0625) -- (\i-0.0625,\j-0.0625) -- (\i+0.0625,\j-0.0625) -- (\i+0.0625,\j+0.0625);
}
}

\foreach \i in {4.375,4.5,...,5.625,5.75,5.875,6}{
\foreach \j in {0.375}{
\fill[black] (\i,\j) circle (0.015cm);
\draw (\i+0.0625,\j+0.0625) -- (\i-0.0625,\j+0.0625) -- (\i-0.0625,\j-0.0625) -- (\i+0.0625,\j-0.0625) -- (\i+0.0625,\j+0.0625);
}
}

\foreach \i in {4.375,4.5,...,5.375,5.5}{
\foreach \j in {0.25}{
\fill[black] (\i,\j) circle (0.015cm);
\draw (\i+0.0625,\j+0.0625) -- (\i-0.0625,\j+0.0625) -- (\i-0.0625,\j-0.0625) -- (\i+0.0625,\j-0.0625) -- (\i+0.0625,\j+0.0625);
}
}

\foreach \i in {4.5,4.625,4.75,...,5.25,5.375}{
\foreach \j in {0.125}{
\fill[black] (\i,\j) circle (0.015cm);
\draw (\i+0.0625,\j+0.0625) -- (\i-0.0625,\j+0.0625) -- (\i-0.0625,\j-0.0625) -- (\i+0.0625,\j-0.0625) -- (\i+0.0625,\j+0.0625);
}
}

\foreach \i in {4.875,5,5.125,5.25,5.375}{
\foreach \j in {0}{
\fill[black] (\i,\j) circle (0.015cm);
\draw (\i+0.0625,\j+0.0625) -- (\i-0.0625,\j+0.0625) -- (\i-0.0625,\j-0.0625) -- (\i+0.0625,\j-0.0625) -- (\i+0.0625,\j+0.0625);
}
}

\foreach \i in {4.25}{
\foreach \j in {0.25,0.375,0.5}{
\fill[black] (\i,\j) circle (0.015cm);
\draw (\i+0.0625,\j+0.0625) -- (\i-0.0625,\j+0.0625) -- (\i-0.0625,\j-0.0625) -- (\i+0.0625,\j-0.0625) -- (\i+0.0625,\j+0.0625);
}
}

\end{tikzpicture}

\caption{Sketch of the situation in Proposition \ref{prop: BV}. The open set $A \subseteq \Omega$ is colored in gray, the set $\Omega' \subset\subset A$ is colored in blue. The points in $I_{\delta}$ are indicated by black dots. The corresponding cubes $q_i^{\delta}$ are sketched with black boundaries. One specific of the larger cubes $Q_i^{\delta}$ is sketched in red. Note that they cover $\Omega'$ for $\delta>0$ small enough. The function $R_{\delta}$ is constant on each of the cubes $q_i^{\delta}$. Hence, $DR_{\delta}$ is concentrated on the faces of $\partial q_i^{\delta}$.  }
\label{fig: coveringbycubes}
\end{figure}
 
 Now, fix $i \in I_{\delta}$. If $n=2$ by the generalized rigidity estimate from \cite{MSZ14} there exists $R_i \in SO(2)$ such that 
 \begin{equation}\label{eq: rigidity n=2}
   \int_{ Q_i^{\delta}} |R - R_i|^2  \, dx \leq C |\operatorname{Curl } R|(Q_i^{\delta})^{2} = C |\operatorname{Curl } R|(Q_i^{\delta})^{\frac{n}{n-1}}.
 \end{equation}
 If $n>2$ we use the generalized rigidity estimate from \cite{LL17} to obtain $R_i \in SO(n)$ such that
 \begin{equation} \label{eq: rigidity n>2}
  \|R - R_i\|_{L^{\frac{n}{n-1},\infty}(Q_i^{\delta})}^{\frac{n}{n-1}} \leq C |\operatorname{Curl } R|(Q_i^{\delta})^{\frac{n}{n-1}}.
 \end{equation}
 Here $L^{\frac{n}{n-1},\infty}$ denotes the weak $L^{\frac{n}{n-1}}$-space which can be constructed as a real interpolaton spaces of the regular $L^p$-spaces via the $K$-method, see, for example, \cite{Lu09}. 
Note that by a scaling argument it can be shown that for all $\delta>0$ and $i\in I_{\delta}$ for $C>0$ one can use the constant for the domain $(0,1)^n$. In particular, $C$ in the inequality above does not depend on $\delta$ nor $i$. 

We define a function $R_{\delta}: \Omega' \to SO(n)$ by $R_{\delta}(x) = R_i$ if $x \in q_i^{\delta}$ where $i \in I_{\delta}$ (note that while each $R_i$ is defined on $Q_i^\delta$ which overlap for neighboring `$i$'s, the smaller cubes $q_i^{\delta}$ are mutually disjoint).
It follows that $R_{\delta} \in BV(\Omega';SO(n))$ and the distributional derivative of $R_{\delta}$ is concentrated on the boundaries of neighboring cubes $q_i$, namely
\begin{equation} \label{eq: jumpsetRdelta}
 |D R_{\delta}|(\Omega') = \sum_{i,j \in I_{\delta}, |i_1 - j_1| + |i_2 - j_2| = \delta} |R_i - R_j| \, \mathcal{H}^{n-1}( \partial q_i \cap \partial q_j \cap \Omega') . 
\end{equation}
 Next, we fix two neighboring indices $i,j \in I_{\delta}$ i.e., $|i_1-j_1| + |i_2 -j_2| = \delta$. Then we use \eqref{eq: rigidity n=2} to find for $n=2$ that
 \begin{align*}
  2 \cdot \delta^2 |R_i - R_j|^2 &= \int_{Q_i^{\delta} \cap Q_j^{\delta}} |R_i - R_j|^2 \, dx  \\
  &\leq 2 \left( \int_{Q_i^{\delta}} |R - R_i|^2 \, dx + \int_{Q_j^{\delta}} |R - R_j|^2 \, dx \right)  \\
  &\leq 2 C \left( |\operatorname{Curl } R|(Q^{\delta}_i)^{2} + |\operatorname{Curl } R|(Q^{\delta}_j)^{2} \right).
 \end{align*}
 In particular, we obtain 
 \[
  |R_i - R_j| \leq C' \delta^{-1} \left( |\operatorname{Curl } R|(Q^{\delta}_i) + |\operatorname{Curl } R|(Q^{\delta}_j) \right).
 \]
Similarly, one proves for $n>2$ using \eqref{eq: rigidity n>2} that
\[
  |R_i - R_j| \leq C' \delta^{-(n-1)} \left( |\operatorname{Curl } R|(Q^{\delta}_i) + |\operatorname{Curl } R|(Q^{\delta}_j)^{2} \right).
 \]
 By the finite overlap of the cubes $Q_i^{\delta}$ we derive from \eqref{eq: jumpsetRdelta} for all $n\geq 2$ that
 \begin{align}
  |D R_{\delta}|(\Omega') &\leq C' \sum_{\substack{i,j \in I^n_{\delta}, \\ |i_1 - j_1| + |i_2 - j_2| = \delta}}  \delta^{-(n-1)} \delta^{n-1} ( |\operatorname{Curl } R|(Q_i) + |\operatorname{Curl } R|(Q_j)) \nonumber \\ 
  &\leq C'' |\operatorname{Curl } R|\left(\bigcup_{i\in I_{\delta}} Q_i \right) \leq C'' \, |\operatorname{Curl } R|(A). \label{eq: lowersemi}
 \end{align}
Moreover, the H{\"o}lder inequality for Lorentz spaces (see \cite{ON63}) yields
 \[
 \int_{Q_i^{\delta}} |R_{i} - R| \, dx \leq C \delta \|R_{\delta} - R\|_{L^{\frac{n}{n-1},\infty}(Q_i^{\delta})}. 
 \]
We then estimate using \eqref{eq: rigidity n>2} and \eqref{eq: rigidity n=2}, respectively,
\begin{align*}
\int_{\Omega'} |R_{\delta} - R| \, dx &= \sum_{i \in I_{\delta}} \int_{\Omega' \cap q_i^{\delta}} |R_i - R| \, dx \\
&\leq \sum_{i \in I_{\delta}} \int_{Q_i^{\delta}} |R_i - R| \, dx \\
&\leq \sum_{i \in I_{\delta}} C \delta \|R_{\delta} - R\|_{L^{\frac{n}{n-1},\infty}(Q_i^{\delta})} \\
&\leq C \delta \, \sum_{i \in I_{\delta}} |\operatorname{Curl }R|(Q_i^{\delta}) \\
&\leq C \delta |\operatorname{Curl }R|(A). 
\end{align*}
For the last inequality we used the finite overlap of the cubes $Q_i^{\delta}$. It follows that $R_{\delta} \to R$ in $L^1(\Omega';\R^{n\times n})$. By the lower-semicontinuity of the total variation we find from \eqref{eq: lowersemi} that $R \in BV(\Omega';\R^{n\times n})$ and
\begin{equation} \label{eq: totalvariationlocal}
|D R|(\Omega')| \leq \liminf_{\delta \to 0} |DR_{\delta}|(\Omega') \leq C'' |\operatorname{Curl } R|(A).
\end{equation}
Note that the constant $C''$ can be chosen independently from $\Omega'$. 

Now we exhaust $A$ by compactly contained open sets. Precisely, we find a sequence of open sets $\Omega'_k \subset \subset A$ such that $\Omega_k' \subseteq \Omega'_{k+1}$ and $\bigcup_{k \in \N} \Omega_k' = A$. Then $DR$ is a vector-valued Radon measure on $A$ and \eqref{eq: totalvariationlocal} yields 
\[
|DR|(A) = \lim_{k \to\infty} |DR|(\Omega_k') \leq C'' |\operatorname{Curl } R|(A).
\]
For $A = \Omega$ it follows immediately that $R \in BV(\Omega;\R^{n\times n})$.

For an arbitrary Borel set $A \subseteq \Omega$, we can find for $\eps > 0$ an open set $A \subseteq O \subseteq \Omega$ such that $|\operatorname{Curl} R|(O) \leq |\operatorname{Curl} R|(A) + \varepsilon$. It follows
\[
|DR|(A) \leq |DR|(O) \leq C'' |\operatorname{Curl}R|(O) \leq C'' \left(|\operatorname{Curl} R|(A) + \varepsilon\right).
\]
Sending $\varepsilon \to 0$ yields \eqref{eq: absolutecontinuitycurl}.
\end{proof}

\begin{remark}\label{remark: Sobolev-regularity}
We note that \eqref{eq: absolutecontinuitycurl} shows that the vector-valued Radon measure $DR$ is absolutely continuous with respect to the Radon measure $|\operatorname{Curl }R|$. In particular, if $\operatorname{Curl } R \in L^1(\Omega;\R^{n\times n \times n})$ then $DR$ is absolutely continuous with respect to the Lebesgue measure. In this case by the Radon-Nikodym Theorem (see \cite[Setion 1.6]{EvGa}) we may write $D R = g \, \mathcal{L}^n$ for some $g \in L^1(\Omega;\R^{n\times n \times n})$ and obtain for almost every $x \in \Omega$
\[
|g(x)| = \lim_{r\to 0} \fint_{B_r(x)} |g(y)| \, dy \leq C \lim_{r\to 0} \fint_{{B_r(x)}} |\operatorname{Curl } R(y)| \, dy = C |\operatorname{Curl } R (x)|.
\]
 In particular, it follows that $\|g\|_{L^1} \leq C \|\operatorname{Curl } R\|_{L^1}$ which implies that directly that $R \in W^{1,1}(\Omega;\R^n)$. In addition, if $\operatorname{Curl } R \in L^{\infty}$ then $R \in W^{1,\infty}$. In this case, by the Sobolev embedding theorem (see, for example, \cite[Theorem 2.4.4]{ziemer}) $R$ can be identified with a function which is locally Lipschitz continuous.
\end{remark}

In light of Remark \ref{remark: Sobolev-regularity} we recall here Rademacher's theorem (see \cite[Theorem 3.1.6]{federer}) which states that every Lipschitz function is differentiable at almost every point.
Next, we show that for a differentiable function $R: \Omega \subseteq \R^n \to SO(n)$ the derivative $D R$ can be expressed in terms of the functions $R$ and $\operatorname{Curl } R$.

\begin{proposition}\label{prop: derivative}
Let $n\in \N$, $\Omega \subseteq \R^n$. Assume that $R: \Omega \to SO(n)$ is differentiable at a point $x\in \Omega$. Then we have for $i,k,l,p \in \{1,\dots,n\}$
\begin{align} \label{eq: derivativecurl1}
2 \left(R(x)^T (\partial_i R)(x)\right)_{kl} = &R(x)_{mk} \left(\operatorname{Curl}R(x)\right)_{mil}
+ R(x)_{mi} \left(\operatorname{Curl}R(x)\right)_{mkl} \\
&+ R(x)_{ml} \left(\operatorname{Curl}R(x)\right)_{mki} \nonumber
\end{align}
and
\begin{align} \label{eq: derivativecurl2}
2 \left((\partial_i R)(x)\right)_{pl} 
= &\left(\operatorname{Curl}R(x)\right)_{pil} 
+ R(x)_{pk} R(x)_{mi} \left(\operatorname{Curl}R(x)\right)_{mkl} \\
&+ R(x)_{pk} R(x)_{ml} \left(\operatorname{Curl}R(x)\right)_{mki}. \nonumber
\end{align}
In particular, we have for $n=3$
\begin{align} \label{eq: representationDR1}
2 \left(R(x)^T (\partial_i R)(x)\right)_{kl} = &\varepsilon_{nil} \left( R(x)^T (\curl R)(x)\right)_{kn}  + \varepsilon_{nkl} \left( R(x)^T (\curl R)(x)\right)_{in} \\
&+ \varepsilon_{nki} \left( R(x)^T (\curl R)(x)\right)_{ln}. \nonumber
\end{align}
and 
\begin{align} \label{eq: representationDR2}
2 \left((\partial_i R)(x)\right)_{pl} = &\varepsilon_{nil} \left( (\curl R)(x)\right)_{pn}  + \varepsilon_{nkl} R(x)_{pk} \left( R(x)^T (\curl R)(x)\right)_{in} \\
&+ \varepsilon_{nki} R(x)_{pk} \left( R(x)^T (\curl R)(x)\right)_{ln}. \nonumber
\end{align}
\end{proposition}

\begin{proof}
Since $R(x) \in SO(n)$ for all $x \in \Omega$ it follows for all $i\in \{1,\dots,n\}$ that $R(x)^T (\partial_i R)(x)$ is skew-symmetric. Consequently we find that
\begin{align*}
2 \left(R(x)^T (\partial_i R)(x)\right)_{kl} =& \left(R(x)^T (\partial_i R)(x)\right)_{kl} - \left(R(x)^T (\partial_i R)(x)\right)_{lk} \\
=& \left[\left(R(x)^T (\partial_i R)(x)\right)_{kl} - \left(R(x)^T (\partial_i R)(x)\right)_{lk} \right] \\
&+ \left[\left(R(x)^T (\partial_k R)(x)\right)_{il} + \left(R(x)^T (\partial_k R)(x)\right)_{li} \right] \\
&- \left[\left(R(x)^T (\partial_l R)(x)\right)_{ki} + \left(R(x)^T (\partial_l R)(x)\right)_{ik} \right] \\
= &\left[\left(R(x)^T (\partial_i R)(x)\right)_{kl} - \left(R(x)^T (\partial_l R)(x)\right)_{ki} \right] \\
&+ \left[\left(R(x)^T (\partial_k R)(x)\right)_{il} - \left(R(x)^T (\partial_l R)(x)\right)_{ik} \right] \\
&+ \left[\left(R(x)^T (\partial_k R)(x)\right)_{li} - \left(R(x)^T (\partial_i R)(x)\right)_{lk} \right] \\
=&  R(x)_{mk} \left(\operatorname{Curl}R(x)\right)_{mil} 
+ R(x)_{mi} \left(\operatorname{Curl}R(x)\right)_{mkl}  \\
&+ R(x)_{ml} \left(\operatorname{Curl}R(x)\right)_{mki}. 
\end{align*}
This shows \eqref{eq: derivativecurl1}. Then \eqref{eq: derivativecurl2} follows immediately by multiplication from the left with $R$. Now, we notice that for $n=3$, it holds $(\operatorname{Curl}R)_{qrs} = \eps_{nrs} (\curl R)_{qn}$. Plugging this identity into \eqref{eq: derivativecurl1} and \eqref{eq: derivativecurl2} yields immediately \eqref{eq: representationDR1} and \eqref{eq: representationDR2}.
\end{proof}

\begin{remark}
We remark that Proposition \ref{prop: derivative} implies that there exists $C>0$ such that for all $R\in C^1(\Omega;SO(n))$ it holds (c.f.~also \cite{NM08})
\[
\| DR \|_{L^{\infty}} \leq C \| \operatorname{Curl } R\|_{L^{\infty}}.
\]
\end{remark}

Combining Proposition \ref{prop: BV}, Remark \ref{remark: Sobolev-regularity} and Proposition \ref{prop: derivative} allows us to prove regularity of rotation fields with a regular $\curl$.

\begin{theorem}
Let $n,k \in \N$, $\Omega \subseteq \R^n$ open and $R: \Omega\to SO(n)$ measurable. Assume that $\operatorname{Curl} R = f$ in the sense of distributions for $f \in C^k(\Omega;\R^{n\times n\times n})$. Then $R \in C^{k+1}(\Omega;\R^{n\times n})$.
\end{theorem}
\begin{proof} First, let $k=0$.
As differentiability is a local property we may assume that $\Omega$ is bounded and $\operatorname{Curl } R$ is bounded. By Remark \ref{remark: Sobolev-regularity} it follows that $R \in W^{1,\infty}(\Omega;\R^{n\times n})$. By the Sobolev-embedding theorem it can hence be identified with a function which is locally Lipschitz-continuous. Then Rademacher's theoerem (see, for example, \cite[Theorem 3.1.6]{federer}) yields that $R$ is differentiable almost everywhere and that at almost every point the classical and the weak derivative coincide. Then Proposition \ref{prop: derivative} implies that the weak derivative $D R$ is for almost every point the sum of terms which are products of components of $R$ and $\curl R$. Thus $DR $ can be represented through a continuous function. This implies that $R \in C^1(\Omega;SO(n))$ which is the statement for $k=0$. If $k > 0$ we can bootstrap this argument. We see now that $DR$ is the sum of products of terms which are $C^1$ (components of $R$) or $C^k$ (components of $\curl R$). Hence, by the product rule $R \in C^2$ and second derivatives are sums of products of $R$, $DR$, $\curl R$, or $D\curl R$. This is the statement for $k=1$. All appearing terms of $DR$ are again $C^1$ if $k\geq 2$. Inductively, one can show that derivatives of order $k+1$ exist and are given by sums of products which consist of components of the first $k$ derivatives of $R$ and $\curl R$.
\end{proof}

An immediate consequence is that rotation fields with a constant $\operatorname{Curl}$ in the sense of distributions are necessarily smooth.

\begin{corollary}\label{cor: regularity}
Let $n \in \N$, $\Omega \subseteq \R^n$ open and $R: \Omega\to SO(n)$ be measurable. Assume that the distributional $\operatorname{Curl} R$ is locally constant. Then $R \in C^{\infty}(\Omega;\R^{n\times n})$.
\end{corollary}

\bibliographystyle{abbrv}
\bibliography{constant_curl}

\end{document}